\documentclass{amsart}

\usepackage{amsmath, amssymb, amsthm, amsfonts, marginnote, stmaryrd}

\usepackage{comment}

\usepackage{xcolor}

\input xy
\xyoption{all}

\usepackage{hyperref}

\swapnumbers
\theoremstyle{plain}
\newtheorem{theorem}[subsubsection]{Theorem}
\newtheorem{lemma}[subsubsection]{Lemma}
\newtheorem{prop}[subsubsection]{Proposition}
\newtheorem{cor}[subsubsection]{Corollary}

\theoremstyle{definition}

\newtheorem{remark}[subsubsection]{Remark}

\def\BB{\mathbb{B}}
\def\CC{\mathbb{C}}

\def\FF{\mathbb{F}}

\def\QQ{\mathbb{Q}}

\def\RR{\mathbb{R}}

\def\ZZ{\mathbb{Z}}

\newcommand\cA{\mathcal{A}}
\newcommand\cB{\mathcal{B}}

\newcommand\cD{\mathcal{D}}
\newcommand\cE{\mathcal{E}}
\newcommand\cF{\mathcal{F}}
\newcommand\cG{\mathcal{G}}
\newcommand\cH{\mathcal{H}}
\newcommand\cI{\mathcal{I}}
\newcommand\cJ{\mathcal{J}}

\newcommand\cM{\mathcal{M}}
\newcommand\cN{\mathcal{N}}

\newcommand\cP{\mathcal{P}}

\newcommand\cR{\mathcal{R}}

\newcommand\cT{\mathcal{T}}

\newcommand\cY{\mathcal{Y}}

\def\bI{\mathbf{I}}

\def\bP{\mathbf{P}}

\newcommand\tilW{\widetilde{W}}

\newcommand\alg{\textup{alg}}
\newcommand\Alg{\textup{Alg}}

\newcommand{\Fl}{\textup{Fl}}

\newcommand\id{\textup{id}}

\newcommand{\ind}{\textup{ind}}

\newcommand\Mod{\textup{Mod}}

\newcommand\res{\textup{res}}

\newcommand\Hom{\textup{Hom}}

\newcommand{\Ad}{\textup{Ad}}

\newcommand\xcoch{\mathbb{X}_*}

\renewcommand\a\alpha
\renewcommand\b\beta
\newcommand\g\gamma
\newcommand\G\Gamma
\renewcommand\d\delta
\newcommand\D\Delta

\newcommand{\ph}{\varphi}
\newcommand{\Sig}{\Sigma}
\newcommand{\s}{\sigma}
\renewcommand{\t}{\tau}

\renewcommand{\L}{\Lambda}
\newcommand{\om}{\omega}
\newcommand{\Om}{\Omega}

\newcommand{\incl}{\hookrightarrow}
\newcommand{\isom}{\stackrel{\sim}{\to}}

\newcommand{\surj}{\twoheadrightarrow}
\newcommand{\mt}{\mapsto}

\newcommand{\wt}[1]{\widetilde{#1}}
\newcommand{\wh}[1]{\widehat{#1}}
\newcommand\quash[1]{}
\newcommand\un{\underline}
\newcommand\ov{\overline}
\newcommand\bs{\backslash}
\newcommand\ot{\otimes}
\newcommand{\op}{\oplus}

\newcommand{\lr}[1]{(\!(#1)\!)}

\newcommand\pt{\textup{pt}}

\newcommand\vn{\varnothing}

\newcommand\xr{\xrightarrow}
\renewcommand\c{\circ}

\newcommand{\Sh}{\mathit{Sh}}

\newcommand{\beq}{\begin{equation*}}
\newcommand{\eeq}{\end{equation*}}

\newcommand{\mon}{\textup{mon}}
\newcommand{\lmon}{\textup{-mon}}

\newcommand{\colim}{\operatorname{colim}}
\newcommand{\Fun}{\operatorname{Fun}}

\newcommand\sss{\subsubsection}

\renewcommand\r{\rho}

\newcommand\bu{\bullet}

 \newcommand\ch{\textup{ch}}

\newcommand\CS{\mathcal{CS}}

\newcommand\opp{\textup{opp}}

\newcommand\sft{\subset_{ft}} 

\newcommand\lmod{\textup{-mod}}

\newcommand\hh{\textit{hh}}

\newcommand{\Morita}{\textup{Morita}}

\newcommand{\Qlbar}{\overline{\mathbb{Q}_{\ell}}}
\newcommand\one{\mathbf{1}}
\renewcommand\Pr{\mathrm{Pr}}

\newcommand\Gp{\textup{Gp}}

\newcommand\AI[1]{\textcolor{blue}{#1}}

\title{Tao-Travkin equivalence revisited}

\author{Penghui Li}
\address{Yau Mathematical Sciences Center\\Tsinghua University\\Beijing, China}
\email{lipenghui@mail.tsinghua.edu.cn}

\author{David Nadler}
\address{Department of Mathematics\\University of California, Berkeley\\Berkeley, CA  94720-3840}
\email{nadler@math.berkeley.edu}

\author{Yakov Varshavsky}
\address{Einstein Institute of Mathematics\\
Edmond J. Safra Campus\\
The Hebrew University of Jerusalem\\
Givat Ram\\ Jerusalem, 9190401, Israel}
\email{yakov.varshavsky@mail.huji.ac.il}

\author{Zhiwei Yun}
\address{Department of Mathematics, MIT,  77 Massachusetts Ave., Cambridge, MA 02139}
\email{zyun@mit.edu}

\dedicatory{}
\date{\today}

\begin{document}

\begin{abstract} We give a simpler proof of the theorem of J.~Tao and R.~Travkin saying that (the neutral block of) the affine Hecke category is the monoidal colimit of finite Hecke categories. The argument is axiomatized to apply to various sheaf-theoretic and group-theoretic contexts. We also prove an extension of the Tao-Travkin equivalence for the whole affine Hecke category as a Morita equivalence.
\end{abstract}

\maketitle

\tableofcontents

%%%%%%%%%%%%%%%%%%%%%%%%%%%%%%%%%%%%%%%%%%%%%%%%%%%%

\section{Statements}

\subsection{Tao--Travkin equivalence}\label{ss:TT} 

Let $W$ be an irreducible affine Weyl group with its usual action on a real vector space $V$ by affine transformations. Let $A_{0}\subset V$ be an alcove stratified by facets. Each facet $F$ of $A_{0}$ has a pointwise stabilizer $W_{F}$ that is a finite Weyl group. The orbifold version of the van Kampen theorem tells us that the inclusions $W_{F}\subset W$ induce an isomorphism
\begin{equation*}
\colim^{\Gp}_{F: \textup{facets of }A_{0}}W_{F}\isom W
\end{equation*}
in the category of groups (in which colimits are denoted $\colim^{\Gp}$). One can also see this isomorphism easily from the Coxeter group presentation of $W$.  Similar isomorphisms can be easily proved when $W$ is replaced by its affine Hecke algebra. 

The remarkable paper of James Tao and Roman Travkin \cite{TT} proves an analogous statement one categorical level up, i.e., for affine Hecke categories instead of algebras. 

Let us review the Tao--Travkin equivalence in more detail. Let $G$ be a connected reductive algebraic group over an algebraically closed field $k$. Let $LG$ be the loop group of $G$ and $L^{+}G$ the arc group. Fix a Borel subgroup $B\subset G$ and let $\bI\subset L^{+}G$ be the Iwahori subgroup obtained as the preimage of $B$ under the reduction map $L^{+}G\to G$. We have the affine flag variety $\Fl=LG/\bI$. Let $\bI^{u}$ be the pro-unipotent radical of $\bI$, and $T=\bI/\bI^{u}$ the abstract Cartan of $G$. Let $W_{0}$ denote the Weyl group of $G$, $\tilW=\xcoch(T)\rtimes W_{0}$ the extended affine Weyl group, and $W\subset \tilW$ the affine Weyl group generated by simple reflections, the latter in bijection with the set $I$ of affine simple roots of $LG$ with respect to $\bI$. %Let $\tilW=\xcoch(T)\rtimes W$ be the extended affine Weyl group.

Let $\ell$ be a prime different from the characteristic of $k$. Consider the equivariant version of the affine Hecke category of $LG$
\begin{equation}\label{aff Hk}
\cH=\Sh(\bI\bs LG/\bI,\Qlbar).
\end{equation}
It is defined to be the colimit over $w\in \tilW$ of the categories $\cH_{\le w}=\Sh(\bI\bs \Fl_{\le w},\Qlbar)$ of ind-constructible $\Qlbar$-complexes supported on the Schubert varieties $\Fl_{\le w}=\ov{\bI w \bI}/\bI$. The category $\cH$ is equipped with the monoidal structure given by convolution.

A subset $J\subset I$ is called {\em finite type} if the corresponding subgroup $W_J\subset W$ generated by $J$ is finite. We use $J\sft I$ to emphasize that $J$ is of finite type. Let $\cD=\{J\sft I\}$, partially ordered by inclusion. For $J\sft I$, let $\bP_J$ be the corresponding standard parahoric subgroup of $LG$ with pro-unipotent radical $\bP_J^u$ and reductive quotient $L_J$, so that $J$ is identified with the set of simple roots of $L_{J}$ with respect to the Borel subgroup $B_{J}:=\bI/\bP_{J}^{u}$. %The image of $\bI$ under the projection $\bP_{J}\to L_{J}$ gives a Borel subgroup $B_{J}\subset L_{J}$, whose unipotent radical we denote by $N_{J}$.

For $J\sft I$, let $\cH_{J}$ be the Hecke category of $L_{J}$
\begin{equation*}
\cH_{J}=\Sh(B_{J}\bs L_{J}/B_{J},\Qlbar).
\end{equation*}
It can be identified with the monoidal full subcategory of $\cH$ of sheaves supported on $\bI\bs \bP_{J}/\bI$. Varying $J$ in $\cD$, we get a diagram of monoidal categories $\cH_{\bu}: \cD\to \Alg(\Pr^{L})$ which extends to a functor
\begin{equation}\label{Hrhd0}
\cH^{\rhd}_{\bu}: \cD^{\rhd}\to \Alg(\Pr^{L})
\end{equation}
by assigning to the cone point of $\cD^{\rhd}$ the category $\cH$.

\begin{theorem}[Tao--Travkin \cite{TT}]\label{th:TT} Assume $G$ is simply-connected. The diagram \eqref{Hrhd0} exhibits $\cH$ as the colimit $\colim^{\Alg}_{\cD}\cH_{\bu}$ in the category $\Alg(\Pr^{L})$ of monoidal presentable categories. 
%[Removed $E$-linearity since it seems not to be discussed anywhere]}
%$\Alg(\Pr^{L}_{E})$ of monoidal presentable $E$-linear categories.
\end{theorem}

\subsection{Main result}
\sss{Goal and motivation}
In this paper we give a simpler proof of Theorem \ref{th:TT}, which allows several useful generalizations with little to no extra effort. The argument we give for Theorem \ref{th:TT} works for any connected Kac-Moody group in any usual sheaf-theoretic context. Below we axiomatize the structures and properties of $\cH$ that we need in the proof.

The core of the argument here is in fact a special case of the much more general Localization Theorem about categorical representations of loop groups that will appear in the work of Dhillon--Varshavsky--Yang \cite{DVY}. We decided to write this short paper just for the affine Hecke category to illustrate the key idea. 

On the other hand, variants of the Tao--Travkin theorem, mainly involving non-simply-connected groups, will be used in the forthcoming work of three of us \cite{LNY-Coxeter, LNY-Ell} on the elliptic realization of affine character sheaves. We record these variants in this paper as well.

\sss{Axiomatic setting}\label{sss:abs}
Let $(W,I)$ be a Coxeter group, equipped with the Bruhat partial order. We continue to use notations introduced in \S\ref{ss:TT} regarding Coxeter groups. In particular, $\cD$ denotes the set of finite type subsets of $I$, partially ordered by inclusion.

Let $\cH$ be a monoidal stable presentable $\infty$-category with a recollement structure (see \cite[Def.2.4.3]{AMR} under the name {\em stratification}; see also \cite[A.1.2]{LNY}) indexed by the poset $W$. The monoidal structure and the recollement structure on $\cH$ will need to satisfy two compatibility conditions that we now formulate.

For any down-closed subset $Y\subset W$ (i.e., $y\in Y$ and $y'\le y$ implies $y'\in Y$), we let $\cH_{Y}$ denote the ``closed'' subcategory of $\cH$ under the recollement structure corresponding to $Y$. More generally, for any locally closed subset $Y\subset W$ (those of the form $Y_{2}-Y_{1}$ where $Y_{1}\subset Y_{2}$ are down-closed), let $\cH_{Y}=\cH_{Y_{2}}/\cH_{Y_{1}}$, which is independent of the presentation of $Y$ as the difference of two down-closed subsets. In particular for any $w\in W$, $\cH_{w}$ is well-defined.

We assume the following:
\begin{enumerate}
\item For each $J\sft I$ and $Y\subset W$ a down-closed union of right $W_{J}$-cosets, $\cH_{Y}$ is closed under right multiplication by $\cH_{J}:=\cH_{W_{J}}$. In particular, $\cH_{J}$ is a monoidal full subcategory of $\cH$.

\item\label{Hw free} For each $J\sft I$ and $w\in W$ minimal in $wW_{J}$ (so that $\cH_{w}\subset \cH_{wW_{J}}$), the map the right $\cH_{J}$-modules
\begin{equation*}
\cH_{w}\ot_{\cH_{\vn}}\cH_{J}\to \cH_{wW_{J}}
\end{equation*}
is an equivalence.

\end{enumerate}

Varying $J$ in $\cD$, we get a diagram of monoidal categories $\cH_{\bu}: \cD\to \Alg(\Pr^{L})$ which extends to a functor
\begin{equation}\label{Hrhd0.5}
\cH^{\rhd}_{\bu}: \cD^{\rhd}\to \Alg(\Pr^{L})
\end{equation}
by assigning to the cone point of $\cD^{\rhd}$ the category $\cH$. 

\begin{theorem}\label{th:abs} Under the above assumptions, the diagram \eqref{Hrhd0.5} exhibits $\cH$ as the colimit $\colim^{\Alg}_{\cD}\cH_{\bu}$ in the category $\Alg(\Pr^{L})$ of monoidal presentable categories.
\end{theorem}

%By sheaves we mean \Yun{add convention; may use abstract 6 functor formalism}. The coefficient field \Yun{or coefficient ring?} is denoted by $E$.

\subsection{Application: connected Kac-Moody groups}
Let $\cG$ be a connected Kac-Moody group over an algebraically closed field $k$. See \cite{Kumar}, \cite{Mat}. Let $\cB$ be a Borel subgroup of $\cG$ and consider the flag variety $\Fl_{\cG}:=\cG/\cB$ of $\cG$, an ind-scheme over $k$. The affine Hecke category $\cH$ of $\cG$ is defined to be the stable $\infty$-category of (\'etale) $\Qlbar$-sheaves on $\Fl_{\cG}$ with left equivariant structures under $\cB$.  The %$E$-linear 
stable $\infty$-category $\cH$ carries a monoidal structure under the convolution product.

Let $W$ be the Weyl group of $\cG$ with respect to a maximal torus $\cT\subset \cB$, with simple reflections in bijection with the set $I$ of simple roots of $\cG$ with respect to $\cB$. For each $i\in I$, let $s_{i}\in W$ be the corresponding simple reflection. Corresponding to $J\sft I$ we have a parabolic subgroup $\cP_{J}\subset\cG$ containing $\cB$, whose Levi quotient $L_{J}$ is a connected reductive group over $k$. The subvariety $\cP_{J}/\cB\subset \Fl_{\cG}$ canonically isomorphic to the flag variety of $L_{J}$. Let $\cH_{J}\subset \cH$ be the full subcategory of sheaves supported on $\cP_{J}/\cB$, which is closed under convolution and can be identified with the Hecke category of $L_{J}$ as a monoidal category. Varying $J$ in $\cD$, we get a diagram of monoidal categories $\cH_{\bu}: \cD\to \Alg(\Pr^{L})$ which extends to a functor
\begin{equation}\label{Hrhd}
\cH^{\rhd}_{\bu}: \cD^{\rhd}\to \Alg(\Pr^{L})
\end{equation}
by assigning to the cone point of $\cD^{\rhd}$ the category $\cH$. Applying Theorem \ref{th:abs} to this situation, we obtain:

\begin{theorem}\label{th:KM} Let $\cG$ be a connected Kac-Moody group over $k$ and $\cH$ be its Hecke category using $\Qlbar$-sheaves. The diagram $\cH^{\rhd}_{\bu}$ induces an equivalence of monoidal categories
\begin{equation*}
\colim^{\Alg}_{\cD}\cH_{\bu}\isom \cH.
\end{equation*}
\end{theorem}

\begin{remark}
The above setting includes the neutral component $\cG=(LG)^{\c}$ of the loop group  considered in \S\ref{ss:TT} for any reductive $G$. It also includes the neutral components of twisted loop groups, i.e., loop groups of quasi-split reductive groups over $k\lr{t}$.
\end{remark}

\subsection{Extended Hecke categories}\label{sss:ext}
While the setting of \S\ref{sss:abs} is suitable for Hecke categories for connected Kac-Moody groups, in applications we often encounter disconnected ones, such as the loop group of a general reductive group over $k\lr{t}$. To include those, we extend the setup of \S\ref{sss:abs} as follows. 

Let $(W,I)$ still denote a Coxeter group, and keep all the relevant notations. Let $\Om$ be a group acting on $(W,I)$ (i.e., acting on $W$ permuting simple reflections). Let $\tilW=W\rtimes \Om$. Extend the Bruhat partial order on $W$ to $\tilW$ by declaring $w_{1}\om_{1}\le w_{2}\om_{2}$ (where $w_{i}\in W, \om_{i}\in \Om$) if and only if $w_{1}\le w_{2}$ and $\om_{1}=\om_{2}$.

Let $\wt\cH$ be a monoidal stable presentable $\infty$-category with a recollement structure indexed by the poset $W$. We again use the notation $\wt\cH_{Y}$ for any locally closed subset $Y\subset \tilW$. Let $\cH=\wt\cH_{W}$. If $Y\subset W$, both $\cH_{Y}$ and $\wt\cH_{Y}$ are defined and they agree. 

In addition to the two conditions in \S\ref{sss:abs} (which only involves $W$ and $\cH$), we impose further that
\begin{enumerate}
\item[(3)] For each $\om\in \Om$, the functor 
\begin{equation}\label{right om whole}
\wt\cH\ot_{\cH_{\vn}}\wt\cH_{\om}\to \wt\cH.
\end{equation}
induced by convolution is an equivalence, and it respects the recollement structures in the sense that for any down-closed $Y\subset \tilW$, \eqref{right om whole} restricts to an equivalence
\begin{equation}\label{right om}
\wt\cH_{Y}\ot_{\cH_{\vn}}\wt\cH_{\om}\to \wt\cH_{Y\om}.
\end{equation}
This then induces an equivalence of the form \eqref{right om} for any locally closed $Y\subset \tilW$.
%The same is true for the right convolution
%\begin{equation}\label{left om}
%\wt\cH_{\om}\ot_{\cH_{\vn}}\wt\cH\to \wt\cH
%\end{equation}
\end{enumerate}

Consider the category $\cD/\Om$: its objects are the same as that of $\cD$, i.e., $J\sft I$. The morphisms are defined as
\begin{equation*}
\Hom_{\cD/\Om}(J_{1},J_{2})=\{\om \in\Om|\om(J_{1})\subset J_{2}\}.
\end{equation*}

Let $\Morita(\Pr^{L})$ be the $(\infty,1)$-category where objects are the same as $\Alg(\Pr^{L})$ but morphisms are $\infty$-groupoids of bimodules.

\begin{theorem}\label{th:abs Om} Under the above assumptions, there is a canonical functor
\begin{equation}\label{H Morita}
\wt\cH^{\Morita,\rhd}_{\bu}: (\cD/\Om)^{\rhd}\to \Morita(\Pr^{L}).
\end{equation}
extending $\cH_{\bu}$ and sending the cone vertex to $\wt\cH$. The restriction functor
\begin{equation*}
\rho: \wt\cH\lmod\to \lim_{J\in (\cD/\Om)^{\opp}}\cH_{J}\lmod
\end{equation*}
induced by $\wt\cH^{\Morita,\rhd}_{\bu}$ is an equivalence. \end{theorem}

\begin{remark}\label{r:colim Om} As usual, one can turn the limit in Theorem \ref{th:abs Om} into a colimit by passing to left adjoints (i.e., induction functors) to obtain an equivalence
\begin{equation*}
\colim_{J\in \cD/\Om}\cH_{J}\lmod\isom \wt\cH\lmod
\end{equation*}
where the arrow is induced by the induction functors $\wt\cH\ot_{\cH_{J}}(-): \cH_{J}\lmod\to \wt\cH\lmod$.
\end{remark}

\begin{remark} 
When $\Om=1$ hence $\tilW=W$ is a Coxeter group, Theorem \ref{th:abs} and Theorem \ref{th:abs Om} have the same content; in fact, the first step of the proof of Theorem \ref{th:abs} given in \S\ref{ss:red} is to reduce  it to the corresponding case of 
Theorem \ref{th:abs Om}.
\end{remark}

\subsection{Application: disconnected twisted loop groups}\label{ss:disc}

%For applications to loop groups, it is desirable to have a version of the Tao-Travkin equivalence that allows the loop group $LG$ to be disconnected, i.e., dropping the simply-connectedness assumption on $G$. In this case, we do not know a presentation of $\cH$ as a colimit of monoidal categories in the style of Tao-Travkin; what we offer is a Morita equivalence, i.e., a presentation of the category $\cH\lmod$ as a colimit/limit of module categories over finite Hecke categories.

\sss{Setup}\label{sss:tw loop group}
It would also be useful to include twisted loop groups, i.e., those coming from non-split groups over $k\lr{t}$. We therefore extend the setup of \S\ref{ss:TT} as follows. A systematic discussion of twisted loop groups can be found in \cite{PR}. In this subsection $G$ denotes a connected reductive group over the field $F=k\lr{t}$ that splits over a tamely ramified extension (so the setup of \S\ref{ss:TT} corresponds to split $G$). Let $LG$ be the loop space of $G$ and $\bI$ be an Iwahori subgroup. Let $\tilW$ and $W$ continue to denote the extended affine Weyl group and the affine Weyl group of $LG$. Then $\tilW$ indexes the double cosets $\bI\bs LG/\bI$ while $W$ indexes those contained in the neutral component of $LG$. The affine Weyl group $W$ is a Coxeter group with simple reflections $I$. The quotient $\Om=\tilW/W$ is the inertia covariants of the algebraic fundamental group $\pi_{1}^{\alg}(G_{\ov F})$, and there is a canonical section $\Om\incl  \tilW$ whose image consists of elements that normalize $\bI$ (i.e., length zero elements in $\tilW$). We identify $\Om$ with its image under the canonical embedding and hence write $\tilW=W\rtimes \Om$. The action of $\Om$ on $W$ permutes the set of simple reflections $I$. 

Let $\wt\cH$ be the (extended) affine Hecke category of $LG$, as defined by the right side of \eqref{aff Hk}. The $\bI$-double cosets on $LG$ are indexed by $\tilW$, giving $\wt\cH$ the recollement structure indexed by $\tilW$. Let $\cH=\wt\cH_{W}$ consist of those sheaves in $\wt\cH$ supported on the neutral component of $LG$. Then $\cH$ satisfies the axioms in \S\ref{sss:abs}, and $\wt\cH$ satisfies the extra condition satisfying the conditions in \S\ref{sss:ext}.

Theorem \eqref{th:abs Om} applies to $\wt\cH$ then gives a functor 
\begin{equation*}
\cH^{\Morita,\rhd}_{\bu}: (\cD/\Om)^{\rhd}\to \Morita(\Pr^{L})
\end{equation*}
sending $J$ to $\cH_{J}\lmod$ and the cone vertex to $\wt\cH\lmod$. In this case, Theorem \eqref{th:abs Om} gives the following.
%The construction of this functor will be carried out in \S\ref{ss:Morita}. Note however that $\cH_{\bu}$ does not in general extend to a functor from $(\cD/\Om)^{\rhd}$ to $\Alg(\Pr^{L})$. 

\begin{theorem}\label{th:Om} Let $G$ be any connected reductive group over $k$ with affine Hecke category $\cH$. The restriction functor 
\begin{equation*}
\rho: \cH\lmod\to \lim_{J\in (\cD/\Om)^{\opp}}\cH_{J}\lmod
\end{equation*}
induced by $\cH^{\Morita,\rhd}_{\bu}$ is an equivalence.
\end{theorem}

\begin{remark}
In \S\ref{ss:gen diagram} we prove a further variant of Theorem \ref{th:Om} where we replace $\cD$ by the set of simplices in a suitable subdivision of the fundamental alcove, and replace $\cD/\Om$ by the poset of $\Om$-orbits on simplices. This variant will be applied in \cite{LNY-Ell}.
\end{remark}

%\begin{remark} As usual, one can turn the limit in Theorem \ref{th:Om} into a colimit by passing to left adjoints (i.e., induction functors) to obtain an equivalence
%\begin{equation*}
%\colim_{J\in \cD/\Om}\cH_{J}\lmod\isom \cH\lmod
%\end{equation*}
%where the arrow is induced by the induction functors $\cH\ot_{\cH_{J}}(-): \cH_{J}\lmod\to \cH\lmod$.
%\end{remark}

\subsection{Variants}

\sss{Other sheaf-theoretic contexts}\label{sss:sh}
Results of Tao-Travkin type hold true if one replaces  $\ell$-adic sheaves with other sheaf theories $\Sh(-)$ including
\begin{itemize}
\item Betti setting ($k=\CC$), with any commutative ring as coefficients;
\item \'etale setting (arbitrary $k$) with $\FF_{\ell}, \ov\FF_{\ell},\ZZ_{\ell},  \ov\ZZ_{\ell}, \QQ_{\ell}$ or $\Qlbar$-coefficients, $\ell\ne\ch(k)$;
\item Algebraic $D$-module setting for $\ch(k)=0$;
\item \'etale motivic sheaves (arbitrary $k$) with $\QQ$-coefficients.
\end{itemize}
In all these settings, the corresponding Hecke category for a connected Kac-Moody group is a monoidal category satisfying the conditions in the axiomatic setting \S\ref{sss:abs} and its extension \S\ref{sss:ext}, therefore Theorems \ref{th:KM} and Theorem \ref{th:Om} continue to hold in these sheaf-theoretic contexts.

In checking the condition \eqref{Hw free} in \S\ref{sss:abs}, we observe that any sheaf theory $\Sh(-)$ mentioned above has the following property:
\begin{equation}\label{pb equiv}
\mbox{If $\pi: X\to Y$ is a gerbe for a pro-unipotent group, the pullback $\pi^{*}: \Sh(Y)\to \Sh(X)$ is an equivalence.}
\end{equation}
This is well-known for the Betti, \'etale and $D$-module settings; for the motivic setting with $\QQ$-coefficients, see \cite[Lemma~3.12]{Y}.

%%%%%%%%
%Abstractly, following the work of Heyer--Mann \cite{HM}, suppose $D$ is a $6$-functor formalism for the geometric setup $()$ \Yun{fill in }.  \Yun{After this part is sorted out, may state Theorem \ref{th:KM} for general sheaf theory.}
%
%
%
%We assume that:
%\begin{equation}\label{pb equiv}
%\mbox{If $\pi: X\to Y$ is a gerbe for a pro-unipotent group, the pullback $\pi^{*}: D(Y)\to D(X)$ is an equivalence.}
%\end{equation}

\begin{lemma}\label{l:pb} Let $\cG$ be a connected Kac-Moody group with Weyl group $(W,I)$. Suppose the sheaf theory $\Sh(-)$ satisfies assumption \eqref{pb equiv}, then $\cH=\Sh(\cB\bs \cG/\cB)$ satisfies condition \eqref{Hw free} set out in \S\ref{sss:abs}.
\end{lemma}
\begin{proof}
Let $J\sft I$. Let $w\in W$ be the minimal element in its coset $wW_{J}$, and choose a lifting $\dot w$ of $w$. Consider the isomorphism
\begin{equation*}
\cB\bs \cB w\cP_{J}/\cB\cong \cB\bs \cB w\cB\times^{\cB}\cP_J/\cB\cong (\cB\cap \Ad(w)\cB)\bs \cP_J/\cB\cong (\cB\cap \Ad(w)\cB)\bs L_J/B_J.
%(\cB\cap \Ad(w)\cP_{J})\bs \cP_{J}/\cB\cong (\cB\cap \Ad(w)\cP_{J})\bs L_{J}/B_{J}.
\end{equation*}
Here the first isomorphism uses the fact that $w$ is minimal in $wW_{J}$, the second one uses that the left and right $\cB\times \cB$-translation on $\cB w\cB$ has stabilizer isomorphic to $\cB\cap \Ad(w)\cB$ at $\dot w$ under the first projection; $\cB\cap \Ad(w)\cB$ acts on $L_{J}$ through the map
\begin{equation*}
p_{\dot w}: \cB\cap \Ad(w)\cB\xr{\Ad(\dot w^{-1})} \Ad(w^{-1})\cB\cap \cB\incl \cB\surj B_{J}\incl L_J.
\end{equation*}
Thus we have a further projection
\begin{equation*}
\t_{\dot w}: \cB\bs \cB w\cP_{J}/\cB\to B_{J}\bs L_{J}/B_{J}
\end{equation*}
that is a gerbe for the pro-unipotent group $\ker(p_{\dot w})$.

Comparing with the case $w=1$, we get a factorization of $\t_{w}$ as 
\begin{equation*}
\t_{\dot w}: \cB\bs \cB w\cP_{J}/\cB\cong(\cB\cap \Ad(w)\cB)\bs L_{J}/B_{J}\xr{\s_{\dot w,J}}\cB\bs L_{J}/B_{J}\cong \cB\bs \cP_{J}/\cB \xr{\t_{1}} B_{J}\bs L_{J}/B_{J}
\end{equation*}
Here $\s_{\dot w,J}$ is induced by the inclusion $\Ad(\dot w^{-1}): \cB\cap \Ad(w)\cB\incl \cB$. By assumption, both $\t_{1}^{*}$ and $\t_{\dot w}^{*}$ are equivalences after applying the sheaf theory $D$, we see that $\s_{\dot w,J}^{*}: \cH_{J}= \Sh(\cB\bs \cP_{J}/\cB) \to \Sh(\cB\bs \cB w\cP_{J}/\cB)=\cH_{wW_{J}}$ is also an equivalence. By construction, $\s_{\dot w,J}^{*}$ is equivariant with respect to the right action of $\cH_{J}$.  

The same constructions applied to $J=\vn$ yields a map $\s_{\dot w,\vn}:  \cB\bs \cB w\cB/\cB\to \cB\bs \cB/\cB$, which induces an equivalence $\s_{\dot w, \vn}^{*}: \cH_{\vn}\isom \cH_{w}$, equivariant under the right action of $\cH_{\vn}$. These equivalences are compatible with the inclusions, i.e., the following diagram   is commutative
\begin{equation*}
\xymatrix{\cH_{\vn}\ar@{^{(}->}[d]\ar[r]^{\s_{\dot w,\vn}^{*}} & \cH_{w}\ar@{^{(}->}[d]\\
\cH_{J}\ar[r]^{\s_{\dot w,J}^{*}} & \cH_{wW_{J}}}
\end{equation*}
Via the isomorphisms $\s_{\dot w,\vn}^{*}$ and $\s_{\dot w,J}^{*}$, the canonical functor $\cH_{w}\ot_{\cH_{\vn}}\cH_{J}\to \cH_{wW_{J}}$ given by convolution can be identified with the convolution functor $\cH_{\vn}\ot_{\cH_{\vn}}\cH_{J}\to \cH_{J}$, which is clearly an equivalence.
\end{proof}

%\Yun{Consider removing specific statements about motivic setting}
%
%In particular, the abstract main theorem \ref{th:abs} applies to the theory of motivic sheaves, as developed by Ayoub for schemes and extended to stacks by \Yun{fill in ref}. 
%\Nadler{Note that property \eqref{pb equiv} holds for rational coefficients as detailed in \cite[Lemma~3.12]{Y}.}
%
%Let us record the statement here.
%
%
%\begin{theorem}[Motivic Tao--Travkin theorem] Let $\cG$ be a connected Kac-Moody group with Weyl group $(W,I)$. Let $\cH^{\mot}=DA_{c}(\cB\bs \cG/\cB, \L)$ be the stable $\infty$-category of \'etale motivic sheaves on $\cB\bs \cG/\cB$ with $\QQ$-coefficients \Nadler{It seems  the literature is much more developed for rational coefficients... but there is this: https://arxiv.org/abs/1410.6359}, equipped with the monoidal structure given by convolution. Similarly define the motivic finite Hecke categories $\cH^{\mot}_{J}$ for $J\sft I$. Then the inclusions $\cH^{\mot}_{J}\incl \cH^{\mot}$ induce an equivalence in $\Alg(\Pr^{L})$
%\begin{equation*}
%\colim^{\Alg}_{J\in \cD}\cH^{\mot}_{J}\isom \cH^{\mot}.
%\end{equation*}
%\end{theorem}  

\sss{Monodromic sheaves}

Another variant is, instead of considering two-sided $\cB$-equivariant sheaves on a connected Kac-Moody group $\cG$, or, in the setting of \S\ref{ss:disc}, two-sided $\bI$-equivariant sheaves on $LG$, we may consider monodromic sheaves in various senses. To fix notation, consider the situation of a twisted loop group $LG$ as in \S\ref{sss:tw loop group}. Let $A$ be quotient torus of $\bI$, and let $\CS(A, \L)$ be the set of rank one character sheaves on $A$ with $\L$-coefficients (see \cite[Definition 2.11]{DLYZ}) in the Betti, \'etale or $D$-module setting.  We may consider
\begin{enumerate}
\item The universally monodromic Hecke category $\wt\cH_{\mon}$, which is the full subcategory of $\Sh(\bI^{+}\bs LG/\bI^{+})$ consisting of sheaves that are locally constant along $A\times A$-orbits under left and right translation. See \cite[Definition 2.1.7]{LNY} for the Betti setting and \cite[\S4.1.3]{DLYZ} for the \'etale and $D$-module setting (denoted $\cM$ in \cite{DLYZ}; ignore the central extension).

\item When $\L$ is a field, fix a $W_0$-orbit $\Xi\subset \CS(A,\L)$ ($W_0$ is the finite Weyl group of $(G,A)$). We have a full monoidal subcategory ${}_{\wh\Xi}\wt\cH_{\wh\Xi}$ of $\wt\cH_{\mon}$ generated by those sheaves whose left and right monodromy along $A$-orbits lie in $\Xi$ (it is denoted $\cM_{\Xi}$ or $\bigoplus_{\chi,\chi'\in\Xi}D((\bI,\chi\lmon)\bs LG/(\bI,\chi'\lmon))$ in \cite{DLYZ}). When $\L$ is a local ring with residue field $k_{\L}$, one may similarly fix $\Xi\in \CS(A,k_{\L})$ and get a full monoidal subcategory ${}_{\wh\Xi}\wt\cH_{\wh\Xi}$ of $\wt\cH_{\mon}$ (see \cite[Corollary 2.28]{DLYZ}).

\item When $\L$ is a field and $\Xi\subset \CS(A,\L)$ is a $W_0$-orbit, we have the strictly equivariant category ${}_{\Xi}\wt\cH_{\Xi}$ (denoted ${}^{\textup{s}}\cM^{\textup{s}}_{\Xi}$ or $\bigoplus_{\chi,\chi'\in \Xi}D((\bI,\chi)\bs LG/(\bI,\chi'))$ in \cite{DLYZ}). 
\end{enumerate}

Similarly, we have three versions of monodromic Hecke categories for connected Kac-Moody groups, denoted $\wt\cH_{\mon}, {}_{\wh\Xi}\wt\cH_{\wh\Xi}$ and ${}_{\Xi}\wt\cH_{\Xi}$.  It is standard to check, by adapting the proof of Lemma \ref{l:pb} to the monodromic situation, that each version of the monodromic affine Hecke category above satisfy the conditions of \S\ref{sss:abs} and \S\ref{sss:ext}. Therefore, Theorems \ref{th:KM} hold for the categories $\cH_{\mon}, {}_{\wh\Xi}\cH_{\wh\Xi}$ and ${}_{\Xi}\cH_{\Xi}$, and Theorem \ref{th:Om} hold for the categories $\wt\cH_{\mon}, {}_{\wh\Xi}\wt\cH_{\wh\Xi}$ and ${}_{\Xi}\wt\cH_{\Xi}$. 

In the setup of \S\ref{sss:tw loop group}, we may also replace $LG$ by its Kac-Moody central extension $\wt LG$, and consider the corresponding monodromic affine Hecke categories for $\wt LG$ (with $A$ is replaced by $\wt A$, a one-dimensional larger torus containing the Kac-Moody center of $\wt LG$, and $\Xi$ denotes a $\tilW$-orbit in $\CS(\wt A,\L)$).

\sss{Mixed characteristic Hecke categories} Let $G$ be a connected reductive group over a characteristic zero discrete valued complete field $F'$ with residue field $k=\ov k$ of characteristic $p>0$. Assume that $G$ splits over a tamely ramified extension of $F'$. The loop group and affine flag variety of $G$ are defined as perfect algebraic spaces over $k$ by Xinwen Zhu \cite{Zhu}. Working in the \'etale setting, it makes sense to form the affine Hecke category $\wt\cH$ of $G$ with coefficients $\L=\FF_{\ell}, \ov\FF_{\ell}, \ZZ_{\ell},\ov\ZZ_{\ell},\QQ_{\ell}$ or $\Qlbar$. The monoidal category $\cH=\Sh(\bI\bs (LG)^{\c}/\bI,\L)$ satisfies the conditions in \S\ref{sss:abs}, and $\wt\cH$ satisfies the extra condition in \S\ref{sss:ext}.  Therefore Theorem \ref{th:KM} holds for the neutral block $\cH$, and Theorem \ref{th:Om} holds for $\wt\cH$ in the mixed characteristic setting as well.

%%%%%%%%%%%%%
\subsection*{AI disclosure} Any text generated or inspired by AI is marked \AI{blue}.

%%%%%
\subsection*{Acknowledgments}

We thank Sanath Devalapurkar for helpful comments.

PL was partially supported by National Key R \& D Program of China (grant 2024YFA1014700) and NSFC grant \#12522103.  DN was supported by NSF DMS grant \#2401178. YV and ZY were partially supported by the BSF grant 2024252. YV was partially supported by the ISF grant 2889/25. ZY was partially supported by the Simons Investigator grant.

%%%%%

%%%%%%%%%%%%%%%%%
%%%%%%%%%%%%%%%%%
\section{Proofs}

\subsection{Proof of Theorem \ref{th:abs}: first reductions}\label{ss:red} For \S\ref{ss:red}-\S\ref{ss:pf abs}, we are in the setup of \S\ref{sss:abs}.

Let $\Pr^{L}$ (resp. $\Pr^{R}$) be the $\infty$-category of presentable categories with functors that are left adjoints (resp. right adjoints).

For $\cA\in \Alg(\Pr^{L})$, let $\cA\lmod$ be the $\infty$-category (namely, $(\infty,1)$-category) of $\cA$-modules in $\Pr^{L}$. For a monoidal functor $T: \cA_{1}\to \cA_{2}$, we have a pair of adjoint functors
\begin{equation*}
\xymatrix{T_{*}=\cA_{2}\ot_{\cA_{1}}(-): \cA_{1}\lmod\ar@<.5ex>[r] & \ar@<.5ex>[l] \cA_{2}\lmod : T^{*}.}
\end{equation*}
Here $T^{*}$ keeps the underlying categories and restricts the action from $\cA_{2}$ to $\cA_{1}$ along $T$. Using the induction functors we get a functor
\begin{equation*}
\Mod_{*}: \Alg(\Pr^{L})\to \Pr^{L}. 
\end{equation*}
Using the restriction functors we get
\begin{equation*}
\Mod^{*}: \Alg(\Pr^{L})^{\opp}\to \Pr^{R}. 
\end{equation*}

Consider the limit $\lim_{\cD^{\opp}}\cH_\bu\lmod$ (taken in the $\infty$-category of $\infty$-categories), with transition functors given by restriction functors $\cH_{J'}\lmod\to \cH_{J}\lmod$ for $J\subset J'\in \cD$. An object of $\lim_{J\in \cD^\opp}\cH_J\lmod$ is a category $\cN\in \Pr^{L}$ with compatible actions of $\cH_{J}$ for all $J\in \cD^{\opp}$.

Restriction and induction of modules along the inclusions $\cH_{J}\incl \cH$ give a pair of adjoint functors
\begin{equation*}
\xymatrix{\rho^{L}: \lim_{J\in \cD^\opp}\cH_J\lmod\ar@<.5ex>[r] & \ar@<.5ex>[l]\cH\lmod : \rho}
\end{equation*}
Here $\rho$ sends an $\cH$-module $\cM$ to the same underlying presentable category $\cM$ with actions of $\cH_{J}$ by restricting the $\cH$-action. Its left adjoint $\rho^{L}$ sends $\cN\in \Pr^{L}$ with compatible actions of $\cH_{J}$ to
\begin{equation*}
\r^L(\cN)=\colim_{J\in\cD}\cH\ot_{\cH_J}\cN.
\end{equation*}

\begin{theorem}\label{th:unit} Under the assumptions of \S\ref{sss:abs}, the unit map $u:\id \to \r\c\r^{L}$ is an isomorphism.
\end{theorem}

Assuming this for the moment, we finish the proof of Theorem \ref{th:abs}.

\begin{cor}\label{c:rho equiv} The functors $(\r^{L},\r)$ are inverse equivalences.
\end{cor}
\begin{proof}
Theorem \ref{th:unit} implies that $\r^{L}$ is fully faithful. On the other hand, $\r$ is clearly conservative, hence $\r$ is an equivalence. 

%we check that the counit map $c: \r^{L}\c\r\to \id$ is an isomorphism. Since , it suffices to check this after composing with $\r$, which becomes $\r\c\r^{L}\c\r\to \r$. The composition $\r\to 

\end{proof}

\begin{remark} In our experience, it is Corollary \ref{c:rho equiv} rather than Theorem \ref{th:abs} that is used in various applications. For example, Corollary \ref{c:rho equiv} implies that for any $\cH$-bimodule $\cM$, we have a colimit presentation of its Hochschild homology
\begin{equation*}
\colim_{J\in \cD}\hh(\cH_{J}, \cM)\isom\hh(\cH,\cM).
\end{equation*}
This was proved in \cite[Corollary 2.5.5]{LNY}, using only Theorem \ref{th:unit} and not the Tao-Travkin equivalence directly.
\end{remark}

\sss{Proof of Theorem \ref{th:abs}}
Let $\cH'=\colim_{\cD}\cH_{J}$, with monoidal functor $T: \cH'\to \cH$ induced by the diagram $\cH^{\rhd}_{\bu}$. We need to show that $T$ is an equivalence.

By \cite[Cor. 4.8.5.13]{HA}, $\Mod_{*}$ preserves colimits indexed by contractible simplicial sets, therefore $\Mod^{*}$ preserves limits indexed by contractible simplicial sets. Observe that $\cD$ is contractible since it has a minimal element. Thus the natural functor
\begin{equation*}
\r': \cH'\lmod\to \lim_{\cD^{\opp}}\cH_{\bu}\lmod
\end{equation*}
is an equivalence. Precomposing $\r'$ with $T^{*}: \cH\lmod\to \cH'\lmod$, we recover the functor $\r$. Since both $\r'$ and $\r$ (by Corollary \ref{c:rho equiv}) are equivalences, $T^{*}$ is also an equivalence. \AI{Thus the unit map $\id\to T^{*}\c T_{*}$ is an isomorphism. Applying the unit map to the regular module $\cH'\in \cH'\lmod$, we obtain the functor $ \cH'\to T^{*}\c T_{*}(\cH')=\cH\ot_{\cH'}\cH'\cong \cH$, which is the same as $T$. Thus $T$ is an equivalence}.

\subsection{Proof of Theorem \ref{th:unit}}\label{ss:pf abs}
The unit map $u:\id \to \r\c\r^{L}$ applied to $\cN\in \lim_{J\in \cD^\opp}\cH_J\lmod$ is the composition
\begin{equation}\label{N to colim}
\cN\xr{\one_{\cH}\ot \id_{\cN}} \cH\ot_{\cH_{\vn}}\cN\to \cM:=\colim_{J\in \cD}\cH\ot_{\cH_{J}}\cN.
\end{equation}
Here the first functor sends $C\in \cN$ to $\one_{\cH}\ot C$, where $\one_{\cH}$ is the monoidal unit of $\cH$. Thus to show Theorem \ref{th:unit}, it suffices to show that \eqref{N to colim} is an equivalence at the level of underlying categories.

Recall that $\cH$ admits a recollement structure indexed by the poset $W$. 
%given by the stratification of $\Fl_{\cG}$ into $\cB$-orbits, which are indexed by $W$ with the Bruhat partial order $\le$. 
For any two down-closed subsets $Y\subset Y'\subset W$, we have a recollement
\begin{equation*}
\xymatrix{\cH_{Y}\ar[r]  &\ar@<1ex>[l]\ar@<-1ex>[l] \cH_{Y'} \ar[r] & \ar@<1ex>[l]\ar@<-1ex>[l]\cH_{Y'-Y}. }
\end{equation*}
Here all arrows are in $\Pr^L$, i.e., they preserve colimits.

If $Y$ and $Y'$ are stable under right multiplication by $W_{J}$, then the above is a recollement of $\cH_{J}$-modules. By the functoriality of Lurie tensor product, for any $\cN\in \cH_{J}\lmod$, we get a recollement
\begin{equation*}
\xymatrix{\cH_{Y}\ot_{\cH_{J}}\cN\ar[r]  &\ar@<1ex>[l]\ar@<-1ex>[l] \cH_{Y'}\ot_{\cH_{J}}\cN \ar[r] & \ar@<1ex>[l]\ar@<-1ex>[l]\cH_{Y'-Y}\ot_{\cH_{J}}\cN. }
\end{equation*}

Now let $\cN\in \lim_{\cD^{\opp}}\cH_{\bu}\lmod$. For any down-closed subset $Y\subset W$ and $J\subset J'\in \cD$, the right $\cH_{J}$-linear inclusion $\cH_{YW_{J}}\incl \cH_{YW_{J'}}$ induces functors $\cH_{YW_{J}}\ot_{\cH_{J}}\cN\to \cH_{YW_{J'}}\ot_{\cH_{J'}}\cN$, and we can form the colimit in $\Pr^{L}$
\begin{equation*}
\cM[Y]:=\colim_{J\in \cD}\cH_{YW_{J}}\ot_{\cH_{J}}\cN.
\end{equation*}
For $Y\subset Y'$ (both down-closed), the inclusions $\cH_{YW_{J}}\incl \cH_{Y'W_{J}}$ induce functors
\begin{equation*}
t_{Y}^{Y'}: \cM[Y]\to \cM[Y'].
\end{equation*}
Let $\cY$ be the set of finite down-closed subsets of $W$, partially ordered under inclusion.  Using the functors $t_{Y}^{Y'}$ as transition functors we form the colimit $\colim_{Y\in \cY} \cM[Y]$, equipped with a canonical functor to $\cM$
\begin{equation}\label{colim MY}
\colim_{Y\in \cY} \cM[Y]\to \cM.
\end{equation}

\begin{lemma}
The functor \eqref{colim MY} is an equivalence.
\end{lemma}
\begin{proof}
Using $\cH=\colim_{Y\in \cY}\cH_{YW_{J}}$, we write $\cM$ as a double colimit
\begin{equation*}
\cM=\colim_{J\in \cD}\colim_{Y\in \cY}(\cH_{YW_{J}}\ot_{\cH_{J}}\cN).
\end{equation*}
Swapping the order of colimits, we write $\cM$ as
\begin{equation*}
\cM\cong\colim_{Y\in \cY}\colim_{J\in \cD}(\cH_{YW_{J}}\ot_{\cH_{J}}\cN)=\colim_{Y\in \cY} \cM[Y].
\end{equation*}
\end{proof}

Take $Y$ to be the singleton consisting of the unit element $e\in W$, we have $\cM[\{e\}]=\colim_{J\in \cD}\cN=\cN$ (since $\cD$ is contractible), and the functor \eqref{N to colim} can be written as the composition
\begin{equation*}
\cN\cong \cM[\{e\}]\to \colim_{Y\in \cY}\cM[Y]\isom \cM.
\end{equation*}

Thus, in order to show that the functor is an equivalence, and hence finishing the proof of Theorem \ref{th:unit}, it remains to show the following statement. 
\begin{prop}\label{p:}
For any two finite down-closed subsets $Y\subset Y'\subset W$, the transition functor $t_{Y}^{Y'}: \cM[Y]\to \cM[Y']$ is an equivalence.
\end{prop}
\begin{proof}
We extend the partial order on $Y'-Y$ to a total order, and list its elements in increasing order
\begin{equation*}
w_{1}, w_{2}, \cdots, w_{n}.
\end{equation*}
Here increasing implies that for $1\le i<j\le n$,  either $w_{i}<w_{j}$ for the Bruhat order, or they are incomparable. Let $Y_{i}=Y\cup \{w_{1},\cdots, w_{i}\}$ for $i=1,2,\cdots, n$, which is down-closed by construction. We need to show that the transition functors $\cM[Y_{i}]\to \cM[Y_{i+1}]$ is an equivalence. We thus reduce to the case where $Y'-Y$ is a singleton $\{w\}$. 

Let $\cE_{\bu}: \cD\ni J\to \cE_{J}:=\cH_{YW_{J}}\ot_{\cH_{J}}\cN$ and $\cF_{\bu}: \cD\ni J\to \cF_{J}:=\cH_{Y'W_{J}}\ot_{\cH_{J}}\cN$ be functors $\cD\to \Pr^{L}$. We have a morphism in $\Fun(\cD,\Pr^{L})$
\begin{equation*}
t:\cE_{\bu}\to \cF_{\bu}
\end{equation*}
such that the induced map $\colim_{\cD}t: \colim_{\cD}\cE_{\bu}\to \colim_{\cD}\cF_{\bu}$ gives the transition functor $t^{Y'}_{Y}$.

Let $R_{w}$ be the right descent set of $w$, i.e.,
\begin{equation*}
R_{w}=\{i\in I|ws_{i}<w\}.
\end{equation*}
Since $w\ne 1$, $R_{w}\ne \vn$. 

Let $\cD'=\{J\sft I|J\cap R_{w}\ne\vn\}$. Then $\cD'$ is an up-closed subset of $\cD$. The inclusion $\cD'\incl \cD$ is a weak elementary expansion in the sense of \cite[Def. A.4.2]{LNY}. Recall that for a subset $\cI'$ of a poset $\cI$, the inclusion $\cI'\incl \cI$ is called a weak elementary expansion if
\begin{enumerate}
\item[(a)] $\cI\bs \cI'$ contains a unique minimal element which we denote by $0$;
\item[(b)] $\cI'_{>0}=\{x\in \cI'|x>0\}$ is up-closed in $\cD$, and is acyclic (i.e., the geometric realization $|\cI'_{>0}|$ is contractible).
\end{enumerate}
In our situation, $\cD-\cD'$ has a minimal element $\vn\in \cD$ since $R_{w}\ne \vn$. Moreover, $\cD'=\cD'_{>\vn}$ is acyclic by Lemma \ref{l:D' ac} below.

In Lemma \ref{l:expansion}, we check that $t:\cE_{\bu}\to \cF_{\bu}$ is a weak elementary expansion along $\cD'\incl \cD$ in the sense of \cite[Def. A.4.5]{LNY}. By \cite[Prop. A.4.6]{LNY}, we conclude that $\colim_{\cD}\cE_{\bu}\to \colim_{\cD}\cF_{\bu}$ is an equivalence, hence the same is true for $t_{Y}^{Y'}$. 
\end{proof}

\begin{lemma}\label{l:D' ac}
\begin{enumerate}
\item The subset $R_{w}\subset I$ is of finite type.
\item The sub-poset $\cD'=\{J\sft I|J\cap R_{w}\ne \vn\}$ of $\cD$ is acyclic.
\end{enumerate}
\end{lemma}
\begin{proof}
(1) Let $J=R_{w}$ and let $w'\in wW_{J}$ be the minimal length element. Write $w=w'v$ for $v\in W_{J}$. By construction, for any $j\in J=R_{w}$, we have $w'vs_{j}=ws_{j}<w=w'v$. In particular, $\ell(w')+\ell(vs_{j})=\ell(w'vs_{j})<\ell(w)=\ell(w')+\ell(v)$, which implies $\ell(vs_{j})<\ell(v)$. This being true for all $j\in J$, we conclude that $v$ is a maximal element in $W_{J}$, which exists only when $W_{J}$ is finite, i.e., $J$ is of finite type.

(2) Let $\cR$ be the set of non-empty subsets of $R_{w}$, partially ordered by inclusion. Consider the monotone map $\ph: \cD'\to \cR$ sending $J\in \cD'$ to $J\cap R_{w}$. For any $K\in \cR$, we have $\ph^{-1}(\cR_{\supset K})$ has a minimal element $K$, hence acyclic. By Quillen's Theorem A \cite{Q}, $|\ph|: |\cD'|\to |\cR|$ is homotopy equivalence. Now $|\cR|$ is a standard simplex hence contractible. Therefore $|\cD'|$ is also contractible, and $\cD'$ is acyclic. 
\end{proof}

\begin{lemma}\label{l:expansion}
The map $t: \cE_{\bu}\to \cF_{\bu}$ is a weak elementary expansion along $\cD'\incl \cD$ in the sense of \cite[Def.A.4.5]{LNY}: i.e.,
\begin{enumerate}
\item For $J\in \cD'$, the map $t_{J}: \cE_{J}\to \cF_{J}$ is an equivalence.
\item For $J\notin \cD'$, the following a pushout diagram
\begin{equation}\label{EF push}
\xymatrix{\cE_{\vn}\ar[r]^{t_{\vn}}\ar[d] & \cF_{\vn}\ar[d]\\
\cE_{J}\ar[r]^{t_{J}} & \cF_{J}}
\end{equation}
\end{enumerate}

\end{lemma}
\begin{proof}
(1) For $J\in\cD'$, i.e., $s_{i}\in J$ for some simple reflection $s_{i}$ such that $ws_{i}<w$, we have $ws_{i}\in Y$ (since $Y'$ is down-closed), hence we have $wW_{J}\subset YW_{J}$ and $YW_{J}=Y'W_{J}$. This implies the transition functor $t_{J}: \cE_{J}\to \cF_{J}$ is an equivalence. 

(2) For $J\notin \cD'$, we first claim that $wW_{J}\cap Y=\vn$. Indeed, $J\notin \cD'$ means $J\cap R_{w}=\vn$. For any $i\in J$, we then have $ws_{i}>w$. This implies that $w$ is the minimal element in $wW_{J}$. Since $w\notin Y$ and $Y$ is down-closed, we conclude that $wW_{J}\cap Y=\vn$, or $Y'W_{J}=YW_{J}\coprod wW_{J}$. Consider the commutative diagram where the two rows are recollements of $\cF_{\vn}$ and $\cF_{J}$
\begin{equation}\label{map of rec}
\xymatrix{\cH_{Y}\ot_{\cH_{\vn}}\cN \ar[r]^{t_{\vn}}\ar[d] & \cH_{Y'}\ot_{\cH_{\vn}}\cN\ar[r]\ar[d] & \cH_{w}\ot_{\cH_{\vn}}\cN\ar[d]^{\a_{w}}\\
\cH_{YW_{J}}\ot_{\cH_{J}}\cN\ar[r]^{t_{J}} & \cH_{Y'W_{J}}\ot_{\cH_{J}}\cN\ar[r] & \cH_{wW_{J}}\ot_{\cH_{J}}\cN}
\end{equation}
By Assumption \eqref{Hw free} on $\cH$, using that $w$ is minimal in $wW_{J}$, we have $\cH_{w}\ot_{\cH_{\vn}}\cH_{J}\isom \cH_{wW_{J}}$ as right $\cH_{J}$-modules. Tensoring with $\cN$ we see that $\a_{w}:  \cH_{w}\ot_{\cH_{\vn}}\cN\to \cH_{wW_{J}}\ot_{\cH_{J}}\cN$ is an equivalence. By \cite[Cor.A.2.2]{LNY}, the left square in \eqref{map of rec}, which is the same as \eqref{EF push}, is a pushout diagram.

%Note that for any $J\in \cD$, $\cH_{wW_{J}}$ is free of rank one over $\cH_{J}$ as a right module. Moreover, for a free generator $\d_{w}$ for $\cH_{w}$ (for example, we can take the constant sheaf on the Schubert cell $\Fl_{\cG,w}$ when $\cH$ is the version of the Hecke category as in \S\ref{ss:TT}) as a $\cH_{\vn}$-module, its image in $\cH_{wW_{J}}$ also serve as a free generator of $\cH_{wW_{J}}$ as a $\cH_{J}$-module. Using $\d_{w}$ we can identify both $\cH_{w}\ot_{\cH_{\vn}}\cN$ and $\cH_{wW_{J}}\ot_{\cH_{J}}\cN$ with $\cN$, in such a way that the functor $\a_{w}$ becomes the identity on $\cN$. 

\end{proof}

\subsection{Proof of Theorem \ref{th:abs Om}}

\sss{Construction of $\wt\cH^{\Morita,\rhd}_{\bu}$}
The functor $\wt\cH^{\Morita,\rhd}_{\bu}$ sends $J\in \cD$ to $\cH_{J}\lmod$, and sends the cone vertex $\star$ to $\wt\cH$. For a morphism $\om\in\Hom_{\cD/\Om}(J,J')$, i.e., $\om\in\Om$ such that $\om(J)\subset J'$, the corresponding morphism $\cH_{J}\lmod\to \cH_{J'}\lmod$ in the Morita category is given by the $(\cH_{J'},\cH_{J})$-bimodule
\begin{equation*}
\wt\cH_{W_{J'}\om W_{J}}=\wt\cH_{W_{J'}\om}.
\end{equation*}
Here we use that $W_{J'}\om W_{J}=W_{J'}W_{\om(J)}\om=W_{J'}\om$. The unique morphism $J\to \star$ is sent to the $(\wt\cH, \cH_{J})$-bimodule $\wt\cH$ where $\wt\cH$ and $\cH_{J}$ act by left and right convolution. 

For two composable morphisms $\om_{1}: J\to J'$ and $\om_{2}: J'\to J''$ in $\cD/\Om$, convolution gives a canonical isomorphism of $(\cH_{J''}, \cH_{J})$-bimodules
\begin{equation*}
\wt\cH_{W_{J''}\om_{2}}\ot_{\cH_{J'}}\wt\cH_{W_{J'}\om_{1}}\isom \cH_{W_{J''}\om_{2}\om_{1}}.
\end{equation*}
Indeed, using \eqref{right om} we can write $\wt\cH_{W_{J'}\om_{1}}$ as $\cH_{J'}\ot_{\cH_{\vn}}\wt\cH_{\om_{1}}$. Hence convolution gives 
\begin{equation*}
\wt\cH_{W_{J''}\om_{2}}\ot_{\cH_{J'}}\wt\cH_{W_{J'}\om_{1}}\cong\wt\cH_{W_{J''}\om_{2}}\ot_{\cH_{J'}}\cH_{J'}\ot_{\cH_{\vn}}\wt\cH_{\om_{1}}\cong \wt\cH_{W_{J''}\om_{2}}\ot_{\cH_{\vn}}\wt\cH_{\om_{1}}\cong \wt\cH_{W_{J''}\om_{2}\om_{1}}.
\end{equation*}
Similarly, for $\om: J\to J'$ in $\cD/\Om$, convolution gives an isomorphism of $(\wt\cH,\cH_{J})$-bimodules
\begin{equation*}
\wt\cH\ot_{\cH_{J'}}\cH_{W_{J'}\om}\cong \wt\cH\ot_{\cH_{J'}}\cH_{J'}\ot_{\cH_{\vn}}\wt\cH_{\om}\cong \wt\cH\ot_{\cH_{\vn}}\wt\cH_{\om}\cong \wt\cH.
\end{equation*}
These canonical isomorphisms, together with their obvious analogues for chains of composable morphisms in $\cD/\Om$, give the desired functor $\wt\cH^{\Morita,\rhd}_{\bu}$. Higher compatibilities between these isomorphisms of bimodules follow from the monoidal structure of $\wt\cH$.

%More formally, one has to check the higher compatibilities. But this is a standard consequence of the existence of a lax-symmetric monoidal functor of sheaves $\Sh:\Corr(\textup{Sch}_{ft})\to \Pr^L$.

\sss{Proof of Theorem \ref{th:abs Om}}
Consider the category  $\cD^{\rhd}/\Om$, where the action of $\Om$ is trivial on the cone vertex which we denote by $\star_{1}$. For $J\in \cD$, we have $\Hom(J, \star_{1})=\Om$. Form the cone $(\cD^{\rhd}/\Om)^{\rhd}$ and denote the new cone vertex which we denote by $\star$. Clearly there is an embedding $(\cD/\Om)^{\rhd}\incl (\cD^{\rhd}/\Om)^{\rhd}$ induced by the $\Om$-equivariant embedding $\cD\incl\cD^{\rhd}$. We extend the functor $\wt\cH^{\Morita,\rhd}_{\bu}$ to a functor
\begin{equation*}
\wt\cH^{\Morita,\rhd\rhd}_{\bu}: (\cD^{\rhd}/\Om)^{\rhd}\to \Morita(\Pr^{L})
\end{equation*}
as follows. It sends $\star_{1}$ to $\cH\lmod$. Each $\om\in \Hom(J,\star_{1})=\Om$ is sent to the $(\cH,\cH_{J})$-bimodule $\wt\cH_{W\om}$. Each $\om\in\Hom(\star_{1},\star)=\Om$ is sent to the $(\wt\cH,\cH)$-bimodule $\wt\cH$ by left and right convolution. The identity $(J\xr{\om}J'\xr{\om'}\star_{1})=(J\xr{\om'\om}\star_{1})$ of morphisms in $\cD^{\rhd}/\Om$ corresponds to the isomorphism of $(\cH, \cH_{J})$-bimodules 
\begin{equation*}
\wt\cH_{W\om'}\ot_{\cH_{J'}}\cH_{W_{J'}\om}\cong \wt\cH_{W\om'}\ot_{\cH_{J'}}\cH_{J'}\ot_{\cH_{\vn}}\wt\cH_{\om}\cong \wt\cH_{W\om'}\ot_{\cH_{\vn}}\wt\cH_{\om}\cong \wt\cH_{W\om'\om}.
\end{equation*}
We omit giving the other isomorphisms of bimodules corresponding to morphism compositions in $(\cD^{\rhd}/\Om)^{\rhd}$ that involve $\star_{1}$.

Let $\pi: \cD/\Om\to \star_{1}/\Om$, and let $\pi_{!}$ (resp. $\pi_{*}$) denote the left (resp. right) Kan extension along $\pi$. The functor $\wt\cH^{\Morita,\rhd\rhd}_{\bu}$ restricted to $\cD^{\rhd}/\Om$ induces a map in $\Fun(\cD/\Om, \Pr^{L})$
\begin{equation*}
\wt\a^{L}: \wt\cH^{\Morita}_{\bu}\to \pi^{*}\un{\cH\lmod}_{\star_{1}/\Om}
\end{equation*}
where $\un{\cH\lmod}_{\star_{1}/\Om}$ is the restriction of $\wt\cH^{\Morita,\rhd\rhd}_{\bu}$ to $\star_{1}/\Om$, i.e., the object $\cH\lmod$ with $\Om$-action given by the bimodules $\{\wt\cH_{W\om}\}_{\om\in\Om}$. Passing to the right adjoint and using the adjunction $(\pi^{*}, \pi_{*})$ we get a morphism in $\Fun(\star_{1}/\Om,\Morita(\Pr^{L}))$
\begin{equation*}
\wt\a:\un{\cH\lmod}_{\star_{1}/\Om} \to\pi_{*}\wt\cH^{\Morita}_{\bu}.
\end{equation*}
Taking limits over $\star_{1}/\Om$ we obtain
\begin{equation*}
\a: (\cH\lmod)^{\Om}:=\lim_{\star_{1}/\Om}\cH\lmod\to \lim_{J\in(\cD/\Om)^{\opp}}\cH_{J}\lmod.
\end{equation*}
On the other hand, restricting $\wt\cH^{\Morita,\rhd\rhd}_{\bu}$ to $(\star_{1}/\Om)^{\rhd}$ induces a functor
\begin{equation*}
\b^{L}: \colim_{\star_{1}/\Om}\cH\lmod \to \wt\cH\lmod.
\end{equation*}
Passing to the right adjoint, we get
\begin{equation*}
\b: \wt\cH\lmod\to \lim_{\star_{1}/\Om}\cH\lmod=(\cH\lmod)^{\Om}.
\end{equation*}
The composition 
\begin{equation*}
\wt\cH\lmod\xr{\b}(\cH\lmod)^{\Om}\xr{\a}\lim_{J\in(\cD/\Om)^{\opp}}\cH_{J}\lmod
\end{equation*}
is the functor $\r$ in the statement of the theorem, because $\r$ is induced from $\wt\cH^{\Morita,\rhd}_{\bu}$, while $\a\c\b$ is induced from the the restriction of $\wt\cH^{\Morita,\rhd\rhd}_{\bu}$ to $(\cD/\Om)^{\rhd}$, which is exactly $\wt\cH^{\Morita,\rhd}_{\bu}$. Therefore it suffices to show that both $\a$ and $\b$ are equivalences.

%induces morphisms in $\Morita(\Pr^{L})$
%\begin{equation*}
%\colim_{J\in\cD/\Om}\cH_{J}\lmod\xr{\a} \colim_{\star_{1}/\Om}\cH\lmod \xr{\b} \wt\cH\lmod.
%\end{equation*}
%Passing to right adjoints (restriction functors) we get morphisms in $\Morita(\Pr^{L})$
%\begin{equation*}
%\wt\cH\lmod\xr{\b^{R}}\lim_{\star_{1}/\Om}\cH\lmod=(\cH\lmod)^{\Om}\xr{\a^{R}}\lim_{J\in(\cD/\Om)^{\opp}}\cH_{J}\lmod
%\end{equation*}
%whose composition is the functor $\r$ which we would like to show to be an equivalence. 
%
%We claim that both $\a^{R}$ and $\b^{R}$ are equivalences. Let $\pi: \cD/\Om\to \star_{1}/\Om$, and let $\pi_{*}$ denote the right Kan extension along $\pi$. The restriction of $\wt\cH^{\Morita,\rhd\rhd}_{\bu}$ to $\cD^{\rhd}/\Om$ induces a morphism in $\Pr^{L}$
%\begin{equation*}
%\wt\a^{R}:\un{\cH\lmod}_{\pt/\Om} \to\pi_{*}\wt\cH^{\Morita}_{\bu}.
%\end{equation*}
%The first term above means the restriction of $\wt\cH^{\Morita,\rhd\rhd}_{\bu}$ to $\pt/\Om=\star_{1}/\Om$, i.e., the object $\cH\lmod$ with $\Om$-action given by the bimodules $\{\wt\cH_{W\om}\}_{\om\in\Om}$. 

We claim that $\wt\a$ is an equivalence as functors $\star_{1}/\Om\to \Morita(\Pr^{L})$. To check this, it suffices to check it after pullback to $\pt$, which becomes the canonical restriction functor $\cH\lmod\to \lim_{J\in \cD^{\opp}}\cH_{J}\lmod$ induced by \eqref{Hrhd0.5}, and the latter is an equivalence by Corollary \ref{c:rho equiv}. Since $\a$ is obtained from $\wt\a$ by taking limits over $\star_{1}/\Om$, $\a$ is also an equivalence.

Finally we show that $\b$ is an equivalence. Consider the induction-restriction adjunction
\begin{equation*}
\xymatrix{\ind: \cH\lmod\ar@<.5ex>[r] & \ar@<.5ex>[l]\wt\cH\lmod : \res}
\end{equation*}
The monad $\res\c\ind: \cH\lmod\to \cH\lmod$ is given by the $\cH$-bimodule $\wt\cH=\op_{\om\in\Om}\wt\cH_{W\om}$, the summands of the latter is our definition of the $\om\in \Om$-action on $\cH\lmod$. Therefore 
\begin{equation}\label{compare monads}
\Mod_{\res\c\ind}(\cH\lmod)\cong (\cH\lmod)^{\Om}.
\end{equation}
Since $\res$ is conservative, Barr-Beck-Lurie theorem implies that $\res$ induces an equivalence $\wt\cH\lmod\isom\Mod_{\res\c\ind}(\cH\lmod)$, which via \eqref{compare monads} identifies with $\b$. Therefore $\b$ is an equivalence. The proof of Theorem \ref{th:abs Om} is complete.

\subsection{Variant: more general diagrams}\label{ss:gen diagram}

With applications to \cite{LNY-Coxeter} and \cite{LNY-Ell} in mind,  we develop variants of Theorem \ref{th:Om} where one changes the index category $\cD/\Om$ to a poset coming from a subdivision of the fundamental alcove of $LG$.

\sss{Index poset} %The argument of Theorem \ref{th:main coarse} does not use specifics about the barycentric subdivision $\sd(\cD)$ except that is it fine enough so that Lemma \ref{l:left cov} holds.  We now state a more general version of Theorem \ref{th:main coarse} that will be used in \cite{LNY-Coxeter}.

We are in the setup of \S\ref{sss:tw loop group}. Let $A_{0}\subset \xcoch(T)_{\RR}$ be the fundamental alcove corresponding to the Iwahori subgroup $\bI$. The simplicies of $A_{0}$ are in canonical bijection with $\cD$.

Recall that a poly-simplicial complex $\Sigma$ with an action of a discrete group $\Gamma$ is called a {\em $\G$-semiregular} if for any $\g\in \Gamma$, and any closed simplex $\sigma$ of $\Sigma$, the intersection  $\sigma \cap \g \sigma$ is fixed by $\g$ pointwise.

Let $\Sig$ be a $\Om$-semiregular subdivision of the fundamental alcove $A_{0}$. For example, the barycentric subdivision of $A_{0}$ is $\Om$-semiregular. Let $\cD_{\Sig}$ be the {\em opposite} poset of simplices in $\Sig$. Let $\ov{\cD_{\Sig}}$ be the coarse quotient of $\cD_{\Sig}$ by $\Om$: its elements are $\Om$-orbits in $\cD_{\Sig}$ (which are denoted by $[\s]$ for $\s\in \cD_{\Sig}$), and $[\s]\le [\s']$ if and only if there exists representatives $\s$ and $\s'$ in their $\Om$-orbits such that $\s\le \s'$ in $\cD$ (i.e., $\s\supset \s'$). 

Let $p: \cD_{\Sig}\to \ov{\cD_{\Sig}}$ be the natural map of posets and let $p^{\rhd}: \cD_{\Sig}^{\rhd}\to \ov{\cD_{\Sig}}^{\rhd}$ be the induced map on the right cones. 

Let $\cI$ and $\cJ$ be posets. A monotone map $f:\cI\to \cJ$ is called {\em coCartesian} if for any $i\in \cI$ and $f(i)\le j'\in J$, the set $\{i'\in \cI|i\le i',j'\le f(i')\}$ has a unique minimal element, and that element maps to $j'$. It is called a {\em left covering map} if there is a unique $i'\in f^{-1}(j')$ such that $i\le i'$. A left covering map is coCartesian. 

\begin{lemma}\label{l:left cov gen}
The canonical map $p:\cD_{\Sig}\to \ov{\cD_{\Sig}}$ is a left covering map, hence coCartesian. In particular, the induced map $\ov p:\cD_{\Sig}/\Om\to \ov{\cD_{\Sig}}$ is coCartesian.
\end{lemma}
\begin{proof} For $[\s]\le [\s']$ in $\ov{\cD_{\Sig}}$, and $\s\in[\s]$, we need to show that there is a unique $\s'\in [\s']$ such that $\s\le \s'$ in $\cD_{\Sig}$. The existence is clear by the definition of $[\s]\le [\s']$. Now suppose $\s''=\om(\s')\in [\s']$ is another facet with $\s\le \s''$. Then $\s$ contains both $\s'$ and $\s''$, hence $\om^{-1}(\s)\cap \s\supset \s'$.  Since $\Sig$ is $\Om$-semiregular, $\s'$ is fixed pointwise by $\om$, which implies $\s''=\s'$. 
\end{proof}

Let $r: \cD_{\Sig}\to \cD$ be the map sending a simplex $\s$ of $\Sig$ to $J$ where $A_{J}$ is the unique facet containing the interior of $\s$. This is an order-preserving map of posets.

\begin{lemma}\label{l:subdiv cofinal} The map $r:\cD_{\Sig}\to \cD$ is cofinal. 
%In particular, we have a canonical equivalence
%\begin{equation}
%\colim_{\cD_{\Sig}}\cC^{\c\c}_{\Sig}\cong \colim_{\cD}\cC^{\c\c}_{\bu},
%\end{equation}
%and the diagram $\cC^{\c\c,\rhd}_{\Sig}$ exhibits $\cH^{\c}$ as the colimit $\colim_{\cD_{\Sig}}\cC^{\c\c}_{\Sig}$.
\end{lemma}
\begin{proof}
We apply Quillen's Theorem A \cite{Q}. We need to show that for any $J\in \cD$, letting $\cD_{\Sig,\ge J}:=\cD_{\Sig}\times_{\cD}\cD_{\ge J}$ (where $\cD_{\ge J}=\{J'\in \cD|J'\supset J\}$) has contractible geometric realization. Now $|\cD_{\ge J}|$ is homeomorphic to the closure $\ov{A_{J}}$ of the $J$-facet of the $A_{0}$, and  $\cD_{\Sig,\ge J}$ is the set of simplices of $\Sig\cap \ov{A_{J}}$, hence $|\cD_{\Sig,\ge J}|$ is homeomorphic to $\ov{A_{J}}$ as well, which is in particular contractible.
\end{proof}

\sss{Construction of $\wt \cH^{\Morita,\rhd}_{\ov{\cD_{\Sig}},\bu}$}
We define a functor
\begin{equation*}
\wt \cH^{\Morita,\rhd}_{\ov{\cD_{\Sig}},\bu}: \ov{\cD_{\Sig}}\to \Morita(\Pr^{L})
\end{equation*}
as follows.

%\sss{Construction of $\wt X^{\Morita}_{\bu}$ and $\wt\cH_{\bu}$}
%We define a functor
%\begin{equation*}
%\wt X^{\Morita}_{\bu}: \ov{\cD_{\Sig}}\to \Morita(\frS)
%\end{equation*}
%as follows. 

For any simplex $\s$ of $\Sig$, let $\tilW_{\s}$ be the stabilizer of $\s$ under the extended affine Weyl group $\tilW$. Note that $\tilW_{\s}$ is down-closed in $\tilW$. Indeed, we have $\tilW_{\s}=W_{J}\rtimes \Om_{\s}$, where $J=r(\s)$ and $\Om_{\s}$ be the pointwise stabilizer of $\s$ under $\Om$ (it is also the set-theoretic stabilizer of $\s$ because of the $\Om$-semiregularity assumption on $\Sig$). Let
\begin{equation*}
\wt\cH_{\s}:=\wt\cH_{\tilW_{\s}}\subset \wt\cH.
\end{equation*}

%Let $J=r(\s)$, and observe that the normalizer of $\bP_{J}$ under $LG$ fits into an exact sequence
%\begin{equation*}
%1\to \bP_{J}\to N_{LG}(\bP_{J})\to \Om_{J}\to 1.
%\end{equation*}
%Here $\Om_{J}$ is the stabilizer of $J$ under $\Om$. Clearly $\Om_{\s}\subset \Om_{J}$, therefore it makes sense to define $\wt\bP_{\s}$ as the preimage of $\Om_{\s}$ in $N_{LG}(\bP_{J})$. Let $\wt X_{\s}:=\bI\bs \wt\bP_{\s}/\bI$. By construction, the underlying points of $\wt X_{\s}$ are in bijection with $\tilW_{\s}=W_{J}\rtimes\Om_{\s}$, the stabilizer of $\s$ under the extended affine Weyl group $\tilW$. Note that $\tilW_{\s}$ is down-closed in $\tilW$. Let
%\begin{equation}
%\wt\cH_{\s}:=\wt\cH_{\tilW_{\s}}\subset \wt\cH
%\end{equation}

If $\s,\s'\in \cD_{\Sig}$ are in the same $\Om$-orbit, let ${}_{\s'}\tilW_{\s}$ be the transporter of $\s$ to $\s'$ under $\tilW$: it consists of $w\in \tilW$ such that $w(\s)=\s'$, is a right torsor under $\tilW_{\s}$ and a left torsor under $\tilW_{\s'}$. Note that ${}_{\s'}\tilW_{\s}$ is also down-closed: it is equal to $W_{J'}\rtimes {}_{\s'}\Om_{\s}$ where $J'=r(\s')$ and ${}_{\s'}\Om_{\s}$ is the transporter of $\s$ to $\s'$ under $\Om$. The $(\wt\cH_{\s'}, \wt\cH_{\s})$-bimodule $\wt\cH_{{}_{\s'}\tilW_{\s}}$ is free of rank one over $\wt\cH_{\s'}$ and $\wt\cH_{\s}$, hence it induces a canonical equivalence  $\wt\cH_{\s}\lmod\isom \wt\cH_{\s'}\lmod$ in $\Morita(\Pr^{L})$. Moreover, such canonical equivalences carry canonical transitivity isomorphisms. Therefore we may identify all $\wt \cH_{\s}$ for $\s$ in the same $\Om$-orbit as a single object $\wt \cH_{[\s]}\in \Morita(\Pr^{L})$, where $[\s]\in \ov{\cD_{\Sig}}$ is a $\Om$-orbit on $\cD_{\Sig}$. The assignment $[\s]\mt \wt \cH_{[\s]}$ defines $\wt \cH^{\Morita,\rhd}_{\ov{\cD_{\Sig}},\bu}$ at the level of objects: we assign $\wt\cH\lmod$ to the cone vertex. 

Now we define $\cH^{\Morita,\rhd}_{\ov{\cD_{\Sig}},\bu}$ at the level of morphisms. For $[\s]\le [\s']\in \ov \cD/\Om$, we may choose orbit representatives $\s$ and $\s'$ such that $\s\le \s'$. Then $\Om_{\s}\subset \Om_{\s'}$ by the $\Om$-semiregularity of $\Sig$, hence $\tilW_{\s}\subset \tilW_{\s'}$. We thus have an inclusion $\wt \cH_{\s}\subset \wt \cH_{\s'}$ in $\Alg(\Pr^{L})$, which induces a functor $\wt \cH_{[\s]}\lmod\to \wt \cH_{[\s']}\lmod$ by induction. It is easy to check that this morphism in $\Morita(\Pr^{L})$ is independent of the choices of orbit representatives $\s,\s'$ such that $\s\le \s'$. Finally, for the unique morphism $[\s]\to \star$ in $\ov{\cD_{\Sig}}^{\rhd}$, we assign the $(\wt\cH, \wt\cH_{[\s]})$-bimodule $\wt\cH$. This defines $\wt \cH^{\Morita,\rhd}_{\ov{\cD_{\Sig}},\bu}$ at the level of morphisms.

The functor $\wt\cH^{\Morita,\rhd}_{\ov{\cD_{\Sig}},\bu}$ then induces a pair of adjoint functors
\begin{equation}\label{r subdiv}
\xymatrix{\ov{\r_{\Sig}}: \wt\cH\lmod\ar@<-.5ex>[r] & \ar@<-.5ex>[l]\lim_{[\s]\in \ov{\cD_{\Sig}}^{\opp}}\wt\cH_{[\s]}\lmod : \ov{\r_{\Sig}}^{L}}.
\end{equation}
%Using canonical equivalences between $\cH_{\s}\lmod$ and $\cH_{\s'}\lmod$ when $\s$ and $\s'$ are in the same $\Om$-orbit, we can regard an object of $\lim_{[\s]\in \ov\cD_{\Sig}^{\opp}}\wt\cH_{[\s]}\lmod$ as a collection of $\cH_{\s}$-modules $\cN_{\s}$, one for each $\s\in \cD_{\Sig}$, together with isomorphisms for each pair $\s,\s'\in \cD_{\Sig}$ such that $[\s]\le [\s']\in \ov \cD_{\Sig}$
%\begin{equation}\label{lim trans ov D Sig}
%\cH_{\tilW(\s,\s')^{-1}}\ot_{\wt\cH_{\s'}}\cN_{\s'}\cong \cN_{\s}\in \wt\cH_{\s}\lmod
%\end{equation}
%compatible with compositions.

\begin{theorem} In the setup of \S\ref{sss:tw loop group}, let $\Sig$ be an $\Om=\pi^{\alg}_{1}(G)$-semiregular subdivision of the fundamental alcove of $LG$. Then the functors $(\ov{\r_{\Sig}}^{L},\ov{\r_{\Sig}})$ in \eqref{r subdiv} are inverse equivalences. 
\end{theorem}
\begin{proof} Let $\wt\cH^{\Morita,\rhd}_{\cD_{\Sig},\bu}$ be the composition
\begin{equation*}
\wt\cH^{\Morita,\rhd}_{\cD_{\Sig},\bu}:(\cD_{\Sig}/\Om)^{\rhd}\xr{(r/\Om)^{\rhd}} (\cD/\Om)^{\rhd}\xr{\wt\cH^{\Morita, \rhd}_{\bu}} \Morita(\Pr^{L}).
\end{equation*}
It induces a functor
\begin{equation}
\r_{\Sig}: \wt\cH\lmod\xr{\rho}\lim_{J\in (\cD/\Om)^{\opp}}\cH_{J}\lmod\xr{R}\lim_{\s\in (\cD_{\Sig}/\Om)^{\opp}}\cH_{\s}\lmod.
\end{equation}
Here $\r$ is the functor in Theorem \ref{th:abs Om}, which is known to be an equivalence, and $R$ is induced by $r$. By Lemma \ref{l:subdiv cofinal}, $r/\Om:\cD_{\Sig}/\Om\to \cD/\Om$ is cofinal. Therefore $R$ is an equivalence. We conclude that $\r_{\Sig}$ is an equivalence.

%In a completely similar manner to the construction of $\wt\cH^{\Morita, \rhd}_{\bu}$, we define a functor
%\begin{equation*}
%\wt\cH^{\Morita,\rhd}_{\cD_{\Sig},\bu}:(\cD_{\Sig}/\Om)^{\rhd}\to \Morita(\Pr^{L})
%\end{equation*}
%sending $\s\in \cD_{\Sig}$ to $\cH_{\s}\lmod=\cH_{r(\s)}\lmod$ and sending the morphism $\om:\s\to \s'$ in $\cD_{\Sig}/\Om$ (i.e., $\om\in\Om$ with $\om(\s)\le \s'$) to the $(\cH_{\s'}, \cH_{\s})$-bimodule $\wt\cH_{W_{r(\s')\om}}$. 

%The same argument of Theorem \ref{th:abs Om} shows that $\wt\cH^{\Morita,\rhd}_{\cD_{\Sig},\bu}$ induces an equivalence
%\begin{equation}\label{lim DSig stack}
%\wt\cH\lmod\isom\lim_{\s\in (\cD_{\Sig}/\Om)^{\opp}}\cH_{\s}\lmod.
%\end{equation}
%The argument builds on the analog of Theorem \ref{th:abs} for $\cD_{\Sig}$
%\begin{equation*}
%\cH\lmod\isom\lim_{\s\in (\cD_{\Sig})^{\opp}}\cH_{\s}\lmod
%\end{equation*}
%which in turn follows from Theorem \ref{th:abs} and Lemma \ref{l:subdiv cofinal} that says that  $r^{\opp}: (\cD_{\Sig})^{\opp}\to \cD^{\opp}$ is final (we observe that $\wt\cH^{\Morita,\rhd}_{\cD_{\Sig},\bu}$ when restricted to $\cD_{\Sig}$ is the pullback of $\cH^{\Morita}_{\bu}: \cD\to \Alg(\Pr^{L})\to \Morita(\Pr^{L})$). 

Let $\ov p: \cD_{\Sig}/\Om\to \ov{\cD_{\Sig}}$ be the natural projection. We have a map $\wt\cH^{\Morita}_{\cD_{\Sig}, \bu}\to \ov p^{*}\wt\cH^{\Morita}_{\ov{\cD_{\Sig}},\bu}$ (removing $\rhd$ means removing the cone vertex in the source) whose value at $\s\in \cD_{\Sig}/\Om$ is given by the induction functor $\cH_{\s}\lmod\to  \wt\cH_{\s}\lmod$. It therefore induces a map in $\Fun(\ov{\cD_{\Sig}}, \Morita(\Pr^{L}))$
\begin{equation}\label{LKE isom}
\ov p_{!}\wt\cH^{\Morita}_{\cD_{\Sig}, \bu}\to \wt\cH^{\Morita}_{\ov{\cD_{\Sig}},\bu}.
\end{equation}
Here $\ov p_{!}$ denotes the left Kan extension along $\ov p$. We claim that \eqref{LKE isom} is an isomorphism. Indeed, by Lemma \ref{l:left cov gen}, $\ov p$ is coCartesian, hence the value of $\ov p_{!}\wt\cH^{\Morita}_{\cD_{\Sig}, \bu}$ at $[\s]$ is 
\begin{equation*}
\colim_{\t\in [\s]/\Om}\cH_{\t}\lmod.
\end{equation*}
Note that $[\s]/\Om\cong \{\s\}/\Om_{\s}$, therefore the above colimit is the coinvariants $ (\cH_{\s}\lmod)_{\Om_{\s}}$, which is canonically equivalent to $\wt\cH_{\s}\lmod$ by the induction functor (using the same argument in the proof of Theorem \ref{th:abs Om} showing that $\b$ is an isomorphism). 

Taking colimits of both sides of \eqref{LKE isom} we obtain an isomorphism
\begin{equation*}
\colim_{\s\in\cD_{\Sig}/\Om}\cH_{\s}\lmod\isom \colim_{[\s]\in \ov{\cD_{\Sig}}}\wt\cH_{[\s]}\lmod
\end{equation*}
that is compatible with further maps to $\wt\cH\lmod$. Passing to left adjoints, the functor $\r_{\Sig}$ factors as 
\begin{equation*}
\wt\cH\lmod\xr{\ov{\r_{\Sig}}}\lim_{[\s]\in \ov{\cD_{\Sig}}^{\opp}}\wt\cH_{[\s]}\lmod\isom\lim_{\s\in (\cD_{\Sig}/\Om)^{\opp}}\cH_{\s}\lmod.
\end{equation*}
Since $\r_{\Sig}$ is an equivalence, so is $\ov{\r_{\Sig}}$. This finishes the proof.
\end{proof}

\begin{remark}
When $LG$ is not connected, we do not see how to present $\cH$ as a colimit of finite Hecke categories in the category $\Alg(\Pr^L)$. But we can at least observe that the diagram of left adjoints of \eqref{r subdiv} preserve  full subcategories
of rank one free modules and thus  $\rho^L$  restricts to  an equivalence thereof
$\colim_{[\s]\in \ov \cD_{\Sig}}  \BB\wt\cH_{[\s]}\isom \BB\cH$.
\end{remark}

\end{document}